\documentclass[a4paper,11pt]{amsart}
\usepackage{hyperref,latexsym}
\usepackage{enumerate}

\theoremstyle{plain}
\newtheorem{theorem}{Theorem}[section]
\newtheorem{lemma}[theorem]{Lemma}

\theoremstyle{definition}

\newtheorem{example}{Example}
\theoremstyle{remark}

\begin{document}

\title[Rich system in Riemannian geometry]
      {Rich quasi-linear system for integrable geodesic flows on 2-torus}

\date{20 July 2009}
\author{Misha Bialy and Andrey E. Mironov}
\address{M.Bialy, School of Mathematical Sciences, Raymond and Beverly Sackler Faculty of Exact Sciences, Tel Aviv University,
Israel}
\email{bialy@post.tau.ac.il}
\address{A.E. Mironov, Sobolev Institute of Mathematics and Novosibirsk
State University, 630090 Novosibirsk, Russia }
\email{mironov@math.nsc.ru}
\thanks{The second author (A.M) was partially supported by grants
MK-5430.2008.1 of the President of Russian Federation
and Sibirian Branch of RAS (the interdisciplinary integration
project N. 65)}

\subjclass[2000]{ }

\begin{abstract}
Consider a Riemannian metric on two-torus. We prove that the
question of existence of polynomial first integrals leads naturally
to a  remarkable system of quasi-linear equations which turns out to
be a Rich system of conservation laws. This reduces the question of
integrability to the question of existence of smooth (quasi-)
periodic solutions for this Rich quasi-linear system.

\end{abstract}

\maketitle

\section{Introduction}
\label{sec:intro} In this paper we study the problem of existence of
integrals which are homogeneous polynomials with respect to momenta
for geodesic flows on 2-torus. We show that this problem is
equivalent to the problem of finding periodic solutions for a
remarkable quasi-linear system of equations. We shall prove that in
a domain of hyperbolicity this system can be represented in Riemann
invariants, and also have a form of conservation laws. Such systems
are called the rich systems (or semi-hamiltonian). It is a
challenging problem to show that a given system of size greater than
2 do not have smooth solutions. For Rich systems,  the classical
analysis by Lax along characteristics can be used to establish the
shock formation provided the so-called genuine non-linearity of
eigenvalues is satisfied (see \cite{serre} and also \cite{B1} where
this analysis was performed for a problem of integrability for a
system of 1,5 degrees of freedom). However in our case this genuine
nonlinearity condition does not necessarily holds. This leaves a
hope to find new examples of non-trivial integrals for geodesic
flows on 2-torus. It would be interesting to study if the
generalized hodograph method in the region of hyperbolicity
\cite{tsarev} could give a non-trivial information on smooth
solutions for this system. We hope to come back to these questions
in subsequent paper.

Consider a geodesic flow on the 2-torus ${\mathbb T}^2={\mathbb
R}^2/\mathbb{Z}^2$. Let
$$ds^2=\sum_{i,j=1}^2 g_{ij}(q)dq^idq^j,H=\sum_{i,j=1}^2
g^{ij}(q)p_ip_j$$
 be a Riemannian metric and the corresponding Hamiltonian function
 of the geodesic flow. The geodesic flow is called integrable, if
the Hamiltonian system
$$
 \dot{q}^j=\frac{\partial H}{\partial p_j},\  \dot{p}_j=-\frac{\partial H}{\partial
 q^j},j=1,2,\eqno{(1)}
$$
on $T^*{\mathbb T}^2\simeq {\mathbb T}^2\times{\mathbb R}^2,$ admits
a smooth function $F(q,p)$ (called a first integral) (functionally
independent with $H$ almost everywhere) which has constant values
along the trajectories of the geodesic flow, i.e.
$$
 \{F,H\}=\sum_{j=1}^2\left(\frac{\partial
 F}{\partial q^j}\frac{\partial
 H}{\partial p_j}-\frac{\partial
 H}{\partial q^j}\frac{\partial
 F}{\partial p_j}\right)=0.
$$
We shall denote by $\tau$ the time variable along the trajectories
and $\frac {d}{d\tau}$ will be denoted (untraditionally) by dot
while the variable $t$ will be used later as a coordinate on
2-torus. In this paper we shall deal with the case when $F$ is a
homogeneous polynomial of certain degree with $C^3$ coefficients.

There are two classically known classes of metrics with integrable
geodesic flows. They are written in a conformal coordinates as
follows:

1. $ds^2=f(x)(dx^2+dy^2)$,

2. $ds^2=(f(x)+g(y))(dx^2+dy^2)$

In the first case the Hamiltonian system admits a one parametric
group of symmetries and has the integral which is a polynomial of
the first degree with respect to momenta. While in the second case
the integral appears to be of the second degree and is related to
separation of variables in Hamilton-Jacobi equation. The existence
of the metrics on ${\mathbb T}^2$ with the integrable geodesic flows
having polynomial integrals of the degree higher than 2 and
non-reducible to the integrals of the 1 and 2 degree, is not known.
Amazingly there exist non-trivial examples of geodesic flows on
2-sphere with integrals which are homogeneous polynomials of degrees
3 and 4. These examples (see \cite{selivanova} and \cite{dullin})
were inspired by integrable cases of Goryachev-Chaplygin and
Kovalevskaya in rigid body dynamics (see \cite{bolsinov}).

It is important to mention that in the classical examples 1,2 as
well as in the examples on the sphere mentioned above the metric is
presented in conformal coordinates. This approach goes back at least
to Darboux, and was studied extensively (see for example
\cite{hall}, \cite{kol}, \cite{topalov}) We the reader refer to a
nice exposition \cite {perelomov}.

The main idea of our approach is to work in different coordinates
which are angle coordinates associated to an invariant torus of the
geodesic flow. The advantage of these coordinates is a possibility
to write the quasi-linear equations on the coefficients of the
unknown integral and the metric in the form of evolution equations
as it is described in the following:

\begin{theorem}
\label{one} Suppose that the Hamiltonian system (1) has an integral
$F$, which is a homogeneous polynomial of degree $n$. Then on the
covering plane  ${\mathbb R}^2$ there exist the global coordinates
$(t,x)$, where the metric has the following form
$$
 ds^2=g^2(t,x)dt^2+dx^2,
$$
and the integral $F$ can be written in the form
$$
F=\sum_{k=0}^n \frac{a_{k}(t,x)}{g^{n-k}}p_1^{n-k}p_2^{k},
$$
Where the last two coefficients can be normalized to be
$a_{n-1}\equiv g$ and $a_n\equiv 1.$ Then the commutation relation
$\{F,H\}=0$ is equivalent to the system of $n$ quasi-linear
equations on the unknown $U=(a_0,\dots,a_{n-2},a_{n-1})^{T}$
(remember $ a_{n-1}\equiv g$ and $a_n\equiv1$):
$$
 U_t+A(U)U_x=0,\eqno{(2)}
$$
where the matrix $A$ has the form:
$$
 A=  \left(
  \begin{array}{cccccc}
   0 & 0 & \dots &
  0 & 0 & a_1  \\
  a_{n-1} &
 0 & \dots & 0 & 0 & 2a_2-na_0\\0 &
 a_{n-1} & \dots & 0 & 0 & 3a_3-(n-1)a_1\\
 \dots & \dots & \dots & \dots & \dots & \dots \\
 0 &
 0 & \dots & a_{n-1} & 0 & (n-1)a_{n-1}-3a_{n-3}\\
 0 &
 0 & \dots & 0 & a_{n-1} & na_n-2a_{n-2}\\
  \end{array}\right).\eqno{(3)}
$$
The functions $a_i,g$ are periodic on the variable $x$,
and quasi-periodic on the variable $t$.

\end{theorem}
Let us recall that a hyperbolic diagonal system
$$
 (r_i)_t+\lambda_i(r_1,\dots,r_n)(r_i)_x=0, i=1,..,n
$$ is called Rich if the eigenvalues satisfy the following
conditions
$$\partial_{r_k}\left(\frac{\partial_{r_i}\lambda_j}{\lambda_i-\lambda_j}\right)=
\partial_{r_i}\left(\frac{\partial_{r_k}\lambda_j}{\lambda_k-\lambda_j}\right).$$
It can be proved that a diagonal system is Rich if and only if it
can be written in some coordinates as a system of conservation laws
(we refer to \cite {serre} for a nice exposition of rich systems
including the blow up analysis and to \cite{dubrovin},\cite{tsarev}
for Hamiltonian formalism for system of Hydrodynamics type).

For our system the following theorem holds

\begin{theorem}
The system (2) is the Rich quasi-linear system system. More
precisely:

1. In the region of hyperbolicity (all eigenvalues are real and
distinct) there exists a change of variables (Riemann invariants)
$(a_0,\dots,a_n)\rightarrow(r_1,\dots,r_n)$ transforming the system
to a diagonal form:
$$
 (r_i)_t+\lambda_i(r_1,\dots,r_n)(r_i)_x=0, i=1,..,n.
$$

2. There exists a regular change of variables
$(a_0,\dots,a_n)\rightarrow (G_1,\dots,G_n)$ such that
$G_i,i=1,..,n$ are conservation laws:
$$
 (G_i(a_0,\dots,a_n))_t+(H_i(a_0,\dots,a_n))_x=0, i=1,..,n .
$$

\end{theorem}

We shall see that both Riemann invariants and Conservation laws of
the theorem are ultimately related to the phase portrait of the
geodesic flow.

There are no general methods for proof of existence or non-existence
of smooth periodic solutions of the system of quasi-linear equations
where the eigenvalues can collide and become complex. However in
\cite{B1}, \cite{B2} non-existence of nontrivial periodic solutions
for the so-called Benney chain is proved. It relies heavily on the
property of genuine non-linearity of smallest and largest
eigenvalues. It seems however that  for our system (2) the genuine
non-linearity condition doesn't necessarily holds. This gives a hope
to find smooth (quasi-) periodic solutions for the system.

Let us mention that, starting at $t=0$ with a periodic in $x$
initial data for the Cauchy problem which is hyperbolic. Then one
can "solve" the initial value problem and thus to obtain a metric
$g$ and the integral $F$ of the geodesic flow on a cylinder
$\mathbb{S}^1(x)\times\{ |t|<\epsilon \}$. Of course this is local
result in $t$ and this metric is not necessarily complete. It would
be very interesting to understand this "catastrophe" in more
details.

In the Section 1 we prove the first theorem. We split the proof of
the second theorem: in the Sections 2 and 3 we prove respectively,
that the system (2) can be represented in the Riemann invariants and
in the form of conservation laws.

\section*{Acknowledgements}
We would like to thank  Sergey Petrovich Novikov  and Oleg Mohov for
their interest in the subject of this paper and encouragement. We
are also thankful to Steve Schochet and Marshall Slemrod for
interesting discussions on quasi-linear systems. We are grateful to
Tel Aviv University for hospitality and excellent working
conditions.

\section{Fermi coordinates. Proof of Theorem 1.1}
The metric of the form $
 ds^2=g^2(t,x)dt^2+dx^2
$ is of course very well known in Riemannian geometry. It can always
be introduced locally near a given geodesic. In such a case the
coordinates $(t,x)$ are called usually Fermi coordinates or
sometimes equi-distant coordinates. It turns out that for integrable
geodesic flow these coordinates can be chosen globally. This is done
with the help of a regular invariant torus of the flow which
projects diffeomorphically on the base. The existence of such a
torus is rather deep fact. More precisely, it is proved in \cite{B3}
(Theorem 1.6) that if the geodesic flow on 2-torus has a
non-constant polynomial integral, then there exists an invariant
Lagrangian torus $L\subset T^*{\mathbb T}^2$, such, that the mapping
$$
 \pi|_L:L\rightarrow {\mathbb T}^2
$$
is a diffeomorphism, where $\pi$  is the canonical projection of the
cotangent bundle. Moreover this torus lies in a regular level $F=f$
and it consists of periodic trajectories. One can assume that $f\neq
0$ by perturbing a little the level if it happened to be zero.
Multiplying $F$ by a constant we can achieve $f=1$, this we shall
assume in the sequel. Let $\varphi_1,\varphi_2$ denote the "angle"
part of the action-angle coordinates on $L$ (see \cite{arnold}). In
these coordinates both Hamiltonian flows on $L$, $f^{\tau}$ of $F$
and $h^{\tau}$ of $H$ are linearized. Therefore, we can assume
without loss of generality that the Hamiltonian fields of $H,F$ on
$L$ equal in these coordinates to: $${\rm sgrad} H=(0,1),\quad {\rm
sgrad} F=(\omega_1,\omega_2),$$ where $\omega_1$, $\omega_2$ are
some frequencies. Since $
 \pi|_L:L\rightarrow {\mathbb T}^2
$ is a diffeomorphism, we take $\varphi_1,\varphi_2$ as coordinates
on the base ${\mathbb T}^2$. We shall keep for them the same
notations. The following lemma holds

\begin{lemma}
In the coordinates $\varphi_1,\varphi_2$ the metric $ds^2$ has the
form
$$
 ds^2=a(\varphi_1,\varphi_2)d\varphi_1^2+2bd\varphi_1d\varphi_2+d\varphi_2^2,
$$
where $b\in{\mathbb R}$ is a constant and $a(\varphi_1,\varphi_2)$
is a positive periodic function.

\end{lemma}

\begin{proof} Let us assume that in the coordinates $\varphi_1,\varphi_2$
the metric $ds^2$ has the form
$$
 ds^2=a(\varphi_1,\varphi_2)d\varphi_1^2+2b(\varphi_1,\varphi_2)d\varphi_1d\varphi_2+c(\varphi_1,\varphi_2)d\varphi_2^2.
$$
According to our choice of  coordinates $(\varphi_1,\varphi_2)$ the
curves, defined in the parametric form $(\varphi_1=\varphi_{10},
\varphi_2=\varphi_{20}+\tau)$, are geodesics. Since the length of a
velocity vector of these geodesics is one, we get
$c(\varphi_1,\varphi_2)\equiv 1$. Moreover, we have
$$
 <\pi_*({\rm sgrad} H),\pi_*({\rm sgrad} F)>=b\omega_1+\omega_2,
$$
on the other hand it can be easily checked using the homogenuity of
$F$, that on $L$ the following holds:
$$
 <\pi_*({\rm sgrad} H),\pi_*({\rm sgrad} F)>=pdq({\rm sgrad} F)=$$
 $$=p_1\frac{\partial F}{\partial
 p_1}+p_2\frac{\partial F}{\partial p_2}=nF=nf=n.
$$
Here $pdq$ stands for the canonical one form in the cotangent
bundle. Thus $b\omega_1+\omega_2=nf=n$, and consequently,
$b=(n-\omega_2)/\omega_1$ is a constant.

\end{proof}
Let's introduce new coordinates $(t,x)$:
$$
 \varphi_1=t,\ \varphi_2=x-bt.
$$
Therefore, in the new coordinates the metric takes the form
$$
 ds^2=g^2dt^2+dx^2,
$$
where $g^2=a-b^2$. In the coordinates $t,x$ the vector field
$\pi_*({\rm sgrad} F)$ becomes equal to
$(\omega_1,\omega_2+b\omega_1)=(\omega_1,n).$

Write the integral $F$ with respect to these coordinates:
$$
F=\sum_{k=0}^n \frac{a_{k}(t,x)}{g^{n-k}}p_1^{n-k}p_2^{k},
$$
where $p_1, p_2$ are the conjugate momenta to $\dot{t},\dot{x}$
respectively. Then obviously for the torus $L$ we have $L=\{p_1=0,
p_2=1\}$ and so $F|_L=a_n\equiv1$. Moreover, partial derivatives of
$F$ are easily computed on $L$ to be:
$$\frac{\partial F}{\partial p_1}\Big |_L=\frac {a_{n-1}}{g},\quad
\frac{\partial F}{\partial p_2}\Big |_L=na_n=n.$$ Compare these
values with $\pi_*({\rm sgrad} F)=(\omega_1,n),$ we conclude
$$
 \frac {a_{n-1}}{g}=\omega_1.
$$
In addition, we shall assume that $\omega_1=1$. Indeed, this can be
achieved easily by a rescaling the $t$ variable, $t\rightarrow rt$
for some constant $r$. In order to complete the proof of Theorem
\ref{one} one has to write the condition $\{F,H\}=0$ explicitly and
in this way to get system (2). We omit  this computation. Notice
that by the very construction of the coordinates $t,x$ it follows
that for any function $f$ on the torus $\mathbb{T}^2$ it can be
written as a periodic function $f(\varphi_1,\varphi_2)$, which
equals $f(t,x-bt)$. Therefore, $f$ becomes periodic in $x$ and
(-quasi) periodic in $t$. Theorem 1.1 is proved.

\section{Riemann invariants}
In this Section we show that the system (2) can be represented in
Riemann invariants in the hyperbolic domain by change of field
variables $(a_0,\dots,a_n)\rightarrow(r_1,\dots,r_n).$

Let's fix the energy level of the Hamiltonian
$H=\frac{1}{2}(\frac{p_1^2}{g^2}+p_2^2)=\frac{1}{2}$. Assume that
$p_1=g\cos\varphi,\ p_2=\sin\varphi$, where $\varphi$ is an
angular coordinate in the fibre. Then $F=F(t,x,\varphi)$ becomes a
trigonometric polynomial. We have
$$F=F(t,x,\varphi)=\sum_
{k=0}^{n}a_k cos^{n-k}\varphi\  sin^{k}\varphi;\quad
a_{n-1}=g,a_n=1. \eqno{(4)}$$
$$
 \frac{dF}{d\tau}=F_t\dot{t}+F_x\dot{x}+F_{\varphi}\dot{\varphi}=F_t\frac{\cos\varphi}{g}+
 F_x\sin\varphi+F_{\varphi}\dot{\varphi}=0.\eqno{(5)}
$$
\begin{lemma}
The following equality holds
$$
\chi_A(\lambda)=-\frac {g^{n-1}}{\cos^n\varphi}F_{\varphi}(\varphi),
$$
where $\chi_A$ is the characteristic polynomial of the matrix $A$
with the relation $\lambda=g\tan\varphi$.

\end{lemma}

The proof can be easily obtained by a direct verification.
Remarkably this lemma gives a very clear geometric meaning of the
eigenvalues of the matrix $A(U)$ of our system. They correspond
precisely to those $\varphi$ where the derivative $F_{\varphi}$,
i.e where the invariant tori of the geodesic flow are tangent to
the fibres.

With the help of the lemma and the equations (4),(5) one can easily
write the system (2) in the form of Riemann invariants. Let
$\lambda_i, i=1,\dots,n$ be $n$ distinct real eigenvalues of $A(U)$.
Define $$ r_i=F(\varphi_i), \quad
\varphi_i=arctan(\lambda_i/g),\quad
\varphi_i\in[-\frac{\pi}{2};\frac{\pi}{2}]\quad for \quad
i=1,\dots,n.$$ It follows from the lemma that
$F^{\prime}_{\varphi}(\varphi_i)$ vanishes and so by the equation
(5) it follows that $$
 (r_i)_t+\lambda_i(r_i)_x=0,\  i=1,..,n .
$$
Moreover formula (4) implies that $(r_1,\dots,r_n)$ are regular
coordinates, since the Jacobian matrix for this change of variables
equals $$\Big(\frac{\partial r_i}{\partial
a_k}\Big)=\Big(cos^{n-k}\varphi_i \ sin^{k}\varphi_i\Big),\quad
i=1,\dots,n; \ k=0,\dots,n-1
 $$
which is non-degenerate matrix if
$\varphi_i\in[-\frac{\pi}{2};\frac{\pi}{2}]$ are distinct.

\section{Conservation laws}
In this Section we show that the system (2) can be represented in
the form of the conservation laws. The method of proof of this
statement is based on the work \cite{B4}. Let us first rewrite
Hamiltonian system of the geodesic flow as a system with 1,5 degrees
of freedom away from the invariant torus $L$. This can be done as
follows. Any geodesic $\gamma (\tau)$ different from the "vertical"
ones $\{t=const\}$ has to have $\frac{dt}{d\tau}\neq 0$ and so can
be written as a graph $\{x=x(t)\}$. The length functional can be
rewritten in the following way:
$$
 S(\gamma)=\int|\dot{\gamma}|d\tau=\int\sqrt{{g^2\Big(\frac{dt}{d\tau}\Big)^2+
 \Big(\frac{dx}{d\tau}\Big)^2}}d\tau=
 \int \sqrt{(x^\prime)^2+g^2}=\int Ldt,
$$
where $\mathcal{L}(x^\prime,x,t)=\sqrt{(x^\prime)^2+g^2}$ where we
write (untraditionally)
$x^\prime=\frac{dx}{dt},\dot{\gamma}=\frac{d\gamma}{d\tau}$. Then
the momenta variable for $x^\prime$ is $p=\frac {\partial
\mathcal{L}}{\partial x^\prime}= \frac
{x^\prime}{\sqrt{(x^\prime)^2+g^2}}$.

Then Legendre transform for $\mathcal{L}$ is
$\mathcal{H}(p,x,t)=px^\prime-\mathcal{L}=-g\sqrt{1-p^2}$. So
Hamiltonian equations for geodesic flow get the form:

$$x^\prime=\frac{\partial \mathcal{H}}{\partial
p}=\frac{gp}{\sqrt{1-p^2}},$$
$$p^\prime=-\frac{\partial \mathcal{H}}{\partial
x}=g_x\sqrt{1-p^2}.
$$
Notice that in coordinates $(p,x,t)$ formula (4) of the previous
section for $F$  gets the form
$$F(p,x,t)=\sum_{k=0}^n a_k(1-p^2)^{\frac{k}{2}}p^{n-k},\eqno({4^\prime})$$
where the dependence on $(t,x)$ enters through the coefficients
only.

 Assume $L_1$ is any torus which is invariant under the flow given
 as a graph of a function $p=f(t,x)$. Then the invariance
 condition means that the 1-form $f_tdt+f_xdx-dp$ vanishes on the
 Hamiltonian vector field $\rm sgrad\mathcal{H}$. Therefore we have:
 $$f_t+f_x \frac{\partial \mathcal{H}}{\partial p}+\frac{\partial \mathcal{H}}{\partial
x}=0,$$ which is equivalent to
$$\frac{\partial f }{\partial t}+
\frac{\partial }{\partial x}\mathcal{H}(f(t,x),x,t)=0. \eqno(6)$$
Equation (6) implies that every invariant torus which projects
diffeomorphically to the base produces in fact a conservation low
for our quasi-linear system (2). More precisely, let us choose $n$
disjoint tori $L_1,\dots,L_n$ in the neighborhood of the torus $L$,
lying in regular levels $\{F=c_i\}$ such that their projections to
${\mathbb T}^2$ are diffeomorphisms.

Each of these torii $L_i$ is a graph of a function $f_i$, $
 p=f_i(a_0,\dots,a_{n-1})$. They satisfy the equation

$$
 F(a_0,\dots,a_{n-1},p)=c_i.\eqno{(7)}
$$

Notice that $f_i$ depend on $(t,x)$ implicitly through the
coefficients. Using (6) for $f_i$ we get the conservation laws:
$$
 \frac{\partial f_i }{\partial t}+\frac{\partial}{\partial x}(\mathcal{H}(f_i,x,t))=0.
$$
Let's show that the transition from the variables
$a_0,\dots,a_{n-1}$ to the variables $f_1,\dots,f_n$ is a
diffeomorphism. Introduce for convenience
$$u_1=a_0,\dots,u_{n-1}=a_{n-2},u_n=a_{n-1}=g.$$
 We have
$$
 F(u_1,\dots,u_n,f_i)=c_i,
$$
 Therefore
$$
 \frac{\partial F}{\partial u_j}+\frac{\partial F}{\partial
 p}\Big |_{p=f_i}\frac{\partial f_i}{\partial u_j}=0
$$
or equivalently in the matrix form we have:
$$A+B
    \left(
  \begin{array}{ccc}
   \frac{\partial f_1}{\partial u_1} & \dots & \frac{\partial f_1}{\partial u_n} \\
  \dots & \dots & \dots  \\
    \frac{\partial f_n}{\partial u_1} &
  \dots &   \frac{\partial f_n}{\partial u_n}\\
  \end{array}\right)=0.
$$
where
$$A=\left(
  \begin{array}{ccc}
   \frac{\partial F}{\partial u_1}(f_1) & \dots & \frac{\partial F}{\partial u_n}(f_1) \\
  \dots & \dots & \dots  \\
   \frac{\partial F}{\partial u_1}(f_n) &
  \dots &  \frac{\partial F}{\partial u_n}(f_n)\\
  \end{array}\right),\ B=\left(
  \begin{array}{cccc}
   \frac{\partial F}{\partial p}\Big|_{p=f_1} & 0 & \dots & 0 \\
  \dots & \dots & \dots  & \dots\\
   0 & 0 &
  \dots &  \frac{\partial F}{\partial p}\Big |_{p=f_n}\\
  \end{array}\right).$$

Note that the matrix $A=\left(\frac{\partial F}{\partial
u_j}(f_i)\right)$ is essentially Vandermonde matrix, in particular
it is non-degenerate since the torii $L_i$ are disjoint. The
diagonal matrix $B$ in the formula is non-degenerate either, because
the torii $L_i$ were chosen to be regular graphs. Therefore,
$$
 \frac{\partial (f_1,\dots, f_n)}{\partial (u_1,\dots,u_n)}\ne 0.
$$
This completes the proof of Theorem 2.

In practice it's difficult to reverse the equation (7) explicitly.
Therefore, instead of solving of the equation (7) we act in the
following way.
  Torus $L$ has the the equation $p=1$ in the energy level.
  The value of the integral $F$ equals 1 on $L$. For a small parameter
 $\varepsilon$ solve the equation

 $$F=\sum_{k=0}^n a_k(1-p^2)^{\frac{k}{2}}p^{n-k}=1+\varepsilon
$$
by substituting the power series
$$
 p=1-G_2\varepsilon^2-G_3\varepsilon^3\dots
.
$$
Then it follows from (6) that all the coefficients
$G_i(a_0,\dots,a_{n-1})$ are conservation laws.
\begin{example}
Let $n=3$. In this case $U=(a_0,a_1,a_2)^T$, with $a_2\equiv g,
a_3\equiv 1$.
 The matrix $A(U)$ has the form:
 $$A(U)=\left(
  \begin{array}{ccc}
0&0&a_1\\
a_2&0&2a_2-3a_0\\
0&a_2&3a_3-2a_1\\
\end{array}\right)
$$
The integral has the form:
$$F=a_0(1-p^2)^{\frac{3}{2}}+a_1(1-p^2)p+a_2(1-p^2)^{\frac{1}{2}}+
a_3p^3.\eqno(8)$$ Write
$$p=1-G_2\varepsilon^2-G_3\varepsilon^3-G_4\varepsilon^4\dots$$ and
substitute into  the expression for $F$ in (8), and equate to
$1+\varepsilon$. Computing coefficients for
$\varepsilon,\varepsilon^2,\varepsilon^3$ and equating them to
$1,0,0$ respectively one gets the following conservation laws:
$$ G_2=\frac{1}{2g^2},\\
G_3=\frac{3a_3-2a_1}{2g^4},\\
G_4=\frac{9}{8g^4}+\frac{5(3a_3-2a_1)^2}{8g^6}-\frac{a_0}{g^5}.$$
\end{example}
One can prove in general for any $n$, that all the conservation laws
obtained in this way are rational functions of the field variables.
It is not an easy exercise to verify by hands that these are in fact
conservation laws.
\section{Concluding remarks}
1. An intersting direction for further study would be to understand
algebraic properties of the system. Notice that the matrix $A(U)$
depends linearly on the components of $U$, this case was studied
first by Gelfand, Dorfman see \cite {gelfand}.

2. It is important to understand what the analysis along
characteristics can give for smooth solutions of the system.

3. It is not clear to us if the generalized hodograph method can
give a nontrivial information for system (2).

4. Let us mention an interesting connection between real eigenvalues
of the matrix $A(U)$ and the so called separatrix chains introduced
in \cite {B3}. The geometric language seems to be of tight relation
with analytic properties of our system.

\end{document}